\documentclass[11pt]{amsart}
\usepackage{amsmath,amsfonts,amssymb,amsthm,amscd,latexsym,euscript,verbatim,
bbold}
\usepackage[all]{xy}

\makeatletter
\def\section{\@startsection{section}{1}%
  \z@{1.1\linespacing\@plus\linespacing}{.8\linespacing}%
  {\normalfont\Large\scshape\centering}}
\makeatother

\theoremstyle{plain}

\newtheorem*{MT}{Main Theorem}

\newtheorem*{conj*}{Root Groups Conjecture}
\newtheorem*{thm1.2}{(1.2) Theorem}
\newtheorem*{thm1.3}{(1.3) Theorem}
\newtheorem*{thm1.4}{(1.4) Theorem}
\newtheorem*{prop*}{Proposition}
\newtheorem*{thm*}{Theorem}

\def\eroman{\etype{\roman}}
\def\diag{\operatorname{diag}}

\newtheorem{prop}{Proposition}[section]

\newtheorem{thm}[prop]{Theorem}

\newtheorem{lemma}[prop]{Lemma}

\theoremstyle{definition}
\newtheorem{Def}[prop]{Definition}
\newtheorem*{Def*}{Definition}

\newtheorem*{notation*}{Notation}
\newtheorem{remark}[prop]{Remark}

\newcommand{\etype}[1]{\renewcommand{\labelenumi}{(#1{enumi})}}

\newcommand{\la}{\lambda}

\newcommand{\ff}{F}

\newcommand{\zz}{\mathbb{Z}}

\newcommand{\ga}{\alpha}
\newcommand{\gb}{\beta}
\newcommand{\gc}{\gamma}

\newcommand{\gd}{\delta}

\newcommand{\gre}{\epsilon}
\newcommand{\gl}{\lambda}

\newcommand{\gr}{\rho}
\newcommand{\gs}{\sigma}

\newcommand{\gt}{\tau}

\newcommand{\charc}{{\rm char}}

\newcommand{\half}{\textstyle{\frac{1}{2}}}

\numberwithin{equation}{section}

\begin{document}
\title[Primitive axial algebras are of Jordan type]{Primitive axial algebras are of Jordan type}
\author[Louis Rowen, Yoav Segev]
{Louis Rowen$^*$\qquad Yoav Segev}

\address{Louis Rowen\\
         Department of Mathematics\\
         Bar-Ilan University\\
         Ramat Gan\\
         Israel}
\email{rowen@math.biu.ac.il}
\address{Yoav Segev \\
         Department of Mathematics \\
         Ben-Gurion University \\
         Beer-Sheva 84105 \\
         Israel}
\email{yoavs@math.bgu.ac.il}
\thanks{$^*$The first author was supported by the Israel Science Foundation grant 1623/16}

\keywords{axis, axial algebra,  fusion,
idempotent} \subjclass[2010]{Primary: 17A05, 17A15, 17A20 ;
 Secondary: 17A36,
17C27}

\begin{abstract}
The notion of axial algebra is closely
related to $3$-transposition groups, the Monster group and vertex operator algebras.
In this work we  continue our 
previous works and compete the proof that
all algebras generated by a set of primitive axes
not necessarily of the same type
(see the definition in the body of the paper),  are primitive
axial algebras of Jordan type.
\end{abstract}

\date{\today}
\maketitle
\section{Introduction}
In his simplified construction of the Fischer-Griess Monster group,
Conway \cite{C} constructed a commutative algebra which is not associative,
denoted $B^{\sharp}$ in \cite{HRS}, 
such that to each $2A$-involution of the Monster, there corresponds 
a unique idempotent $a,$ such that multiplication by $a$ from
$B^{\sharp}\to B^{\sharp}$ has minimal polynomial
$(t-1)t(t-\frac{1}{4})(t-\frac{1}{32}).$
Conway called these idempotents ``axial vectors''.  Moreover,
the eigenspaces of these vectors have certain multiplicative
properties which were call ``fusion rules'' in \cite{HRS}.

Another interesting motivating example is as follows.  Given a normal
set $D$ of $3$-transpositions in a group $G,$ one can construct
a {\it Fischer space,} and a corresponding {\it Matsuo algebra,} see
\cite{HRS}.  This Matsuo algebra is a commutative, non-associative
algebra generated by idempotents which are in one to one correspondence
with the transpositions in $D.$  Multiplication by these idempotents
have minimal polynomial $(t-1)t(t-\gl),$ with $\gl\notin\{0,1\},$
and again the eigenspaces satisfy certain fusion rules. 

There are also the {\it Majorana algebras} of Ivanov \cite{I}.
In \cite{HRS} the notion of a Majorana algebra was simplified and
generalized, and all the above examples are included in the new
notion, introduced in \cite{HRS}, of {\it axial algebras.}
We refer the reader to the introduction of \cite{HRS}
for more information.
See also  \cite{DPSC} for non-commutative generalizations.

This paper is a continuation of \cite{RoSe1} and \cite{RoSe2}. 
Together with the previous two paper
it may be considered as a start of the classification
of non-commutative axial algebras.
Our main goal being to classify algebras~$A$ generated by two primitive axes,
and, as a consequence we prove that for any algebra generated by primitive axes,
all the axes are necessarily of Jordan type.

Throughout this paper, $A$ is an algebra not necessarily associative or commutative, not
necessarily with a multiplicative unit element, over a field $F$
with~${\charc(F)\ne 2.}$

\begin{enumerate}\eroman
\item 
For $a,b\in A,$ we define the {\it left and right} multiplication maps
 $L_a(b) := a\cdot b $ and $R_a(b) := b\cdot a.$

\item  
We write $A_{\la}(X_a)$ for the eigenspace
of $\la$ with respect to the transformation $X_a$, $X\in \{L,R\},$ 
i.e., $A_{\la}(L_a) = \{ v \in A: a\cdot v = \la v\},$
and similarly for $A_{\la}(R_a).$ Often we just write $A_\la$ for
$A_{\gl}(L_a)$, when $a$ is understood. We write $A_{\gl,\gd}(a)$ for $A_\gl(L_a) \cap A_\gd(R_a).$
\end{enumerate}

From now on, let $a\in A$ be an idempotent, and $\gl,\gd\notin\{0,1\}$ in $\ff$.
\begin{Def}\label{maind}$ $
\begin{enumerate}
\item
 $a$ is a {\it left axis} of {\it type} $\gl$ if 
$L_a $ satisfies a polynomial $p(t)={t(t-1)(t - \gl).}$
The left axis $a$ is {\it primitive} if $A_1 = \ff a$.

\item
A {\it primitive right axis of type $\gd$} is defined similarly.

\item
$a$ is a  {\it primitive axis (2-sided) of type $(\gl,\gd)$} if $a$ is a 
primitive left axis of  type $\gl$ and a primitive right   axis of type $\gd$ and, in
addition, $L_a R_a =  R_a  L_a.$
Thus $A_{1,1} =\ff a =A_1;$ in particular $A_{1,\gd} = A_{\gl,1} =
0$. Hence $A$ decomposes into a
direct sum
\[
A=\overbrace{A_{1,1} \oplus A_{0,0}}^{\text {$++$-part}}\oplus
\overbrace{A_{0,\gd}}^{\text{$+-$-part}} \oplus \overbrace{
A_{\la,0}}^{\text {$-+$-part}}\oplus \overbrace{A_{\gl,\gd
}}^{\text{$--$-part}},
\]
and this is a $\zz_2\times\zz_2$ grading of $A$ (multiplication $(\gre,\gre')(\gr,\gr')$
is defined in the obvious way for $\gre,\gre',\gr,\gr'\in\{+,-\}$).

In this paper, a {\it primitive axis} is a primitive axis of type $(\gl,\gd),$
for some $(\gl,\gd).$

\item
 $a$ is a primitive axis of {\it Jordan type} if moreover
 $A_{\gl,0} = A_{0,\gl}=0.$

\item
$A$ is a primitive axial algebra of Jordan type, if $A$ is generated by
primitive axes of Jordan type (not necessarily of the same type).
\end{enumerate}
\end{Def}
\medskip

 In \cite{RoSe1} and \cite{RoSe2} we showed that if $A$ is an  algebra of dimension $\le 3$ 
generated by two primitive axes $a,b$, then $a$ and
 $b$ have Jordan type (not necessarily the same!), so $A$ is as in part (1) of the Main Theorem below. 
On the other hand, we showed in \cite[Theorem~4.3]{RoSe2} that $\dim A \le 5$,
 but we did not yet have  any examples of dimension 4 or~5. 
The purpose of this paper is to show that there are no such examples. 
\begin{MT}$ $
\begin{enumerate}
\item
If $A$ is generated by two primitive axes $a,b,$ then $\dim(A)\le 3,$ so $A$
is classified in \cite[Theorem C]{RoSe2}, and \cite[Theorems A and B]{RoSe1}.

\item
If $A$ is generated by a set of primitive axes,
then all the primitive axes in $A$ are of Jordan type.
\end{enumerate}
\end{MT}

We recall from \cite{RoSe1} the definition of the Miyamoto involutions
associated with a primitive axis of type $(\gl,\gd).$

\begin{Def}\label{Miya}[Miyamoto involution]$ $
It is easy to check
that any non-trivial $\mathbb Z_2$-grading of $A$ induces an
automorphism of order $2$ of $A.$  Indeed, if $A=A^+\oplus A^-,$
then $y\mapsto y^+-y^-$ is such an automorphism, where $y\in A$ and
$y=y^+ +y^-$ for $ y^+\in A^+, y^-\in A^-.$

So if $a\in A$ is a primitive axis of type $(\gl,\gd),$ then  we
have three such automorphisms of order 2, which, conforming with the literature, we call the {\bf
Miyamoto involutions} {\it associated with $a$}:

\begin{enumerate}\eroman
\item
$\gt_{\gl,a}=\gt_{\gl}: y\mapsto y-2y_{\gl}=\ga_y a+y_{0,0}+y_{0,\gd}-y_{\gl,0}-y_{\gl,\gd}.$

\item
$\gt_{\gd,a}=\gt_{\gd}: y\mapsto y-2\, {}_{\gd}y=\ga_y a+y_{0,0}-y_{0,\gd}+y_{\gl,0}-y_{\gl,\gd}.$

\item
$\gt_{\diag,a}=\gt_{\diag}: y\mapsto \ga_y a+y_{0,0}-y_{0,\gd}-y_{\gl,0}+y_{\gl,\gd}.$
\end{enumerate}
\end{Def}

Recall also
%
\begin{Def}[Fusion rules]
We call the $\zz_2$-grading of the eigenspaces of a left (resp.~right) axis $a$
the left (resp.~right) {\it fusion rules} for $a.$  

See also \cite[subsection 2.1, p.~85]{HRS}.
\end{Def}

\section{two generated primitive axial algebras}
Throughout this section,  $A$ is  an algebra generated by $2$
primitive axes, $a$ of type $(\gl,\gd)$ and $b$ of type
$(\gl',\gd')$.

We recall the 2-sided eigenvector decomposition of
an element $v\in A,$ with respect to the axis $a:$
\[
v = \ga_v a + v_{0,0} + v_{\gl,0} +v_{0,\gd}+ v_{\gl,\gd},\text{ for }\ga_v \in \ff,\text{ where }v_{\mu,\nu}\in A_{\mu,\nu}(a).
\]
and the decomposition  of an element $w\in A,$ with respect to the axis $b:$
\[
w = \gb_w b + w_{0,0} + w_{\gl',0} +w_{0,\gd'}+ w_{\gl',\gd'},\text{ for }\gb_w \in \ff,\text{ where }w_{\mu,\nu}\in A_{\mu,\nu}(b).
\]
(see \cite[Equation~(1.2)]{RoSe2}.)  

Throughout this paper we denote:
\[
c:=b_{0,0},\qquad x:=b_{\gl,0},\qquad y:=b_{0,\gd},\qquad
z:=b_{\gl,\gd}.
\]
and
\[
c':=a_{0,0},\qquad x':=a_{\gl',0},\qquad y':=a_{0,\gd'},\qquad
z':=a_{\gl',\gd'}.
\]

In this section we shall use:
\begin{prop}[\cite{RoSe2} Proposition 4.7]\label{span}
$A=\ff a+\ff c+\ff x+\ff y+\ff z=\ff b+\ff c'+\ff x'+\ff y'+\ff z'.$
\end{prop}

\begin{remark}\label{eg}
Suppose that $a$ and $b$ have the same type $(\gl,\gd),$ and that $\dim(A)\le 3.$ Then either
\begin{enumerate}
\item
$\dim(A)=2,$ so $A=\ff a+\ff b.$ Furthermore $\gl\ne\gd,$  $\gl+\gd=1,$ $ab=\gd a+\gl b$ and $ba=\gd b+\gl a,$ or

\item
$A$ is commutative.
\end{enumerate}
\end{remark}
\begin{proof}
This follows from \cite[Theorem A]{RoSe1} and \cite[Theorems A and B]{RoSe2}.
\end{proof}
Note that by Proposition \ref{span}
we have: 
\begin{equation}\label{4-1}
\text{If $\dim(A)=4,$ then exactly one of $c', x', y', z'$ is $0.$}
\end{equation}
and
\begin{equation}\label{4-2}
\text{If $\dim(A)=4,$ then exactly one of $c, x, y, z$ is $0.$}
\end{equation}

We start with a number of computations.

%
\begin{lemma}\label{bab}$ $
\begin{enumerate}
\item
$\ga_{bab} a+(bab)_{0,0}=\ga_b^2a+ \gl x^2 +\gl z^2 = \ga_b^2a+ \gd y^2 + \gd z^2.$

\item
$(bab)_{\gl,0}=\ga_b \gl^2 x+\gl cx+\gl yz = \gd yz+\gd zy+ \ga_b\gl x.$

\item
$(bab)_{0,\gd}=\gl xz +\gl zx + \ga_b \gd y  = \ga_b\gd^2 y + \gd yc+\gd zx.$

\item
$(bab)_{\gl,\gd}=\ga_b(\gl^2+\gd) z +\gl cz +\gl yx
= \ga_b(\gl +\gd^2) z+\gd z c + \gd yx.$
\end{enumerate}
\end{lemma}
\begin{proof}
We have
\begin{equation*}
\begin{aligned}
b(ab)= &  (\ga_b a+b_{0,0}+b_{\gl,0}+b_{0,\gd}+b_{\gl,\gd})(\ga_ba + \gl b_{\gl,0} +\gl b_{\gl,\gd}) \\
= &  \ga_b^2 a+\ga_b\gl^2b_{\gl,0}+\ga_b\gl^2 b_{\gl,\gd}\\
+&\gl b_{0,0}b_{\gl,0} +\gl  b_{0,0}b_{\gl,\gd}\\
+& \gl b_{\gl,0}^2+ \gl b_{\gl,0}b_{\gl,\gd} \\
+& \ga_b \gd b_{0,\gd}+ \gl b_{0,\gd}b_{\gl,0} + \gl b_{0,\gd}b_{\gl,\gd}\\
+& \ga_b \gd b_{\gl,\gd}+\gl b_{\gl,\gd}b_{\gl,0} +\gl
b_{\gl,\gd}^2.
\end{aligned}
\end{equation*}
And
\begin{equation*}
\begin{aligned}
(ba)b =  & (\ga_ba + \gd b_{0,\gd} +\gd b_{\gl,\gd})(\ga_b a+b_{0,0}+b_{\gl,0}+b_{0,\gd}+b_{\gl,\gd})\\
= &  \ga_b ^2 a+\ga_b \gd^2b_{0,\gd}+\ga_b\gd^2 b_{\gl,\gd}\\
+&\gd b_{0,\gd}b_{0,0} +\gd b_{\gl,\gd} b_{0,0}\\
+&\ga_b \gl b_{\gl,0} + \gd b_{0,\gd}b_{\gl,0}+\gd  b_{\gl,\gd}b_{\gl,0} \\
+& \gd b_{0,\gd}^2 + \gd b_{\gl,\gd}b_{0,\gd}\\
+& \ga_b \gl b_{\gl,\gd}+\gd b_{0,\gd}b_{\gl,\gd} +\gd b_{\gl,\gd}^2,
\end{aligned}
\end{equation*}
which is $b(ab)$, so equating, fusion rules give
\begin{equation}\label{813}
\begin{aligned}
&\ga_b^2a+ \gl b_{\gl,0}^2 +\gl b_{\gl,\gd}^2 = \ga_b^2a+ \gd b_{0,\gd}^2 + \gd b_{\gl,\gd}^2, \text{ or}\\
&\ga_{bab} a+(bab)_{0,0}=\ga_b^2a+ \gl x^2 +\gl z^2 = \ga_b^2a+ \gd y^2 + \gd z^2.
\end{aligned}
\end{equation}
\begin{equation}\label{823}
\begin{aligned}
&\ga_b \gl^2 b_{\gl,0}+\gl b_{0,0}b_{\gl,0}+\gl b_{0,\gd}b_{\gl,\gd} = \gd b_{0,\gd}b_{\gl,\gd}+\gd b_{\gl,\gd}b_{0,\gd}+\ga_b \gl b_{\gl,0},\text{ or}\\
&(bab)_{\gl,0}=\ga_b \gl^2 x+\gl cx+\gl yz = \gd yz+\gd zy+ \ga_b\gl x.
\end{aligned}
 \end{equation}
 \begin{equation}\label{833}
\begin{aligned}
&\gl b_{\gl,0}b_{\gl,\gd} +\gl b_{\gl,\gd}b_{\gl,0} + \ga_b \gd
b_{0,\gd}  = \ga_b\gd^2 b_{0,\gd} + \gd b_{0,\gd}b_{0,0}+\gd
b_{\gl,\gd}b_{\gl,0},\text {or,}\\
&(bab)_{0,\gd}=\gl xz +\gl zx + \ga_b \gd y  = \ga_b\gd^2 y + \gd yc+\gd zx.
\end{aligned}
\end{equation}
 \begin{equation}\label{843}
\begin{aligned}
\ga_b(\gl^2+\gd) b_{\gl,\gd} &+\gl b_{0,0}b_{\gl,\gd} +\gl b_{0,\gd}b_{\gl,0}\\
&= \ga_b(\gl +\gd^2) b_{\gl,\gd}+\gd b_{\gl,\gd} b_{0,0} + \gd b_{0,\gd}b_{\gl,0},
\end{aligned}
\end{equation}
or
\begin{equation}\label{21}
(bab)_{\gl,\gd}=
\ga_b(\gl^2+\gd) z +\gl cz +\gl yx
= \ga_b(\gl +\gd^2) z+\gd z c + \gd yx.
\end{equation}
\end{proof}
 
\begin{lemma}\label{b^2}$ $
\begin{enumerate}
\item
$\ga_b a+c =  \ga _b^2 a +c^2+x^2+y^2+z^2.$

\item
$x=  \ga _b \gl x + x c+ cx+yz+ zy.$

\item
$y=   \ga _b  \gd y + cy +yc+  xz+z x.$

\item
$z= \ga _b(\gl+\gd) z +cz  + zc  +xy + yx.$
\end{enumerate}
\end{lemma}
\begin{proof}
We have
  \begin{equation*}\label{881}
    \begin{aligned} b  = & b^2 \\
= &(\ga _b a +c+x+y+z)^2 \\
= &\ga _b^2 a +c^2+x^2+y^2+z^2\\
+& \ga _b \gd (y+z) + \ga _b \gl (x+z)\\
+&cx+xc+cy+yc+cz+zc\\
+&xy+yx+xz+zx+yz+zy.
\end{aligned}
\end{equation*}
The lemma follows from the fusion rules.
\end{proof}

\begin{lemma}\label{abb}$ $
\begin{enumerate}
\item
We have
\begin{equation}\label{6-91}
\gd(y^2 +z^2) =
\ga_b(\gb_a(1-\gl')+(\gl'-\ga_b))a+\gb_a(1-\gl')c,
\end{equation}
\begin{equation}\label{6-92}
\gd(yz+zy)=\gb_a(1-\gl')x,
\end{equation}
\begin{equation}\label{6-93}
\gd(cy+xz) = (\gb_a(1-\gl')+\gd(\gl'-\ga_b))y,
\end{equation}
\begin{equation}\label{6-94}
\gd(cz+xy)=(\gb_a(1-\gl')+\gd(\gl'-\ga_b-\ga_b\gl))z.
\end{equation}
\item
We have
\begin{equation}\label{6-95}
\gl (x^2+z^2)=\ga_b(\gb_a(1-\gd')+(\gd'-\ga_b))a+\gb_a(1-\gd')c,
\end{equation}
\begin{equation}\label{6-96}
\gl(xc+zy)=(\gb_a(1-\gd')+\gl(\gd'-\ga_b))x,
\end{equation}
\begin{equation}\label{6-97}
\gl (xz+zx)=\gb_a(1-\gd')y,
\end{equation}
\begin{equation}\label{6-98}
\gl(zc+ xy)=(\gb_a(1-\gd')+\gl (\gd'-\ga_b-\gd\ga_b))z.
\end{equation}
\end{enumerate}
\end{lemma}
\begin{proof}
By \cite[Lemma 2.3(3)]{RoSe1},
\[
b(ba)=\gb_a(1-\gl')b+\gl' (ba).
\]
so
\[
b(\ga_b a+\gd b_{0,\gd}+\gd b_{\gl,\gd})=\gb_a(1-\gl')b+\gl' (ba),
\]
or
\begin{equation}\label{eq bba}
\gd b(y+z)=\gb_a(1-\gl')b+(\gl'-\ga_b)(\ga_ba+\gd(y+z)).
\end{equation}
And also,
\[
(ab)b=\gb_a(1-\gd')b+\gd'ab,
\]
so
\begin{equation}\label{eq abb}
(\gl x+\gl z)b=\gb_a(1-\gd')b+(\gd'-\ga_b)(\ga_ba+\gl x+\gl z).
\end{equation}
\medskip

\noindent
(1) From equation \eqref{eq bba} we get
\[
\gd (\ga_ba+c+x+y+z)(y+z)=\gb_a(1-\gl')b+(\gl'-\ga_b)(\ga_ba+\gd y+\gd z).
\]
Hence
\begin{equation}\begin{aligned}
&\gd(y^2+z^2+yz+zy+cy+xz+cz+xy+\ga_b \gl z)\\
&=\gb_a(1-\gl')b+(\gl'-\ga_b)(\ga_ba +\gd y+\gd z)\\
&=\gb_a(1-\gl')(\ga_ba+c+x+y+z)+(\gl'-\ga_b)(\ga_ba +\gd y+\gd z)
\end{aligned}\end{equation}
Matching components via the fusion rules yields (1).
\bigskip

\noindent (2)  
By \eqref{eq abb} we have
\[
(\gl x+\gl z)(\ga_ba+c+x+y+z)=\gb_a(1-\gd')b+(\gd'-\ga_b)(\ga_ba+\gl x+\gl z)
\]
so
\[
\begin{aligned}
&\gl ((x^2+z^2)+xc+zy+(xz+zx)+\gd\ga_bz+zc+xy)\\
&=\gb_a(1-\gd')(\ga_ba+c+x+y+z)+(\gd'-\ga_b)(\ga_ba+\gl x+\gl z).
\end{aligned}
\]
Matching components via the fusion rules yields (2).
\end{proof}

\begin{lemma}\label{6}
Let $\gs:=ab-\gd'a-\gl b.$ Then
\begin{enumerate}
\item
$\gs =\gc a-\gl (b_{0,\gd}+b_{0,0})=\gr b-\gd'(a_{0,0}+a_{\gl',0}),$
where
\[
\gc=\ga_b(1-\gl)-\gd',\text{ and }\gr=\gb_a(1-\gd')-\gl.
\]

\item
$\gl(c^2+y^2)=\ga_b(\gc-\gr)a-\gr c.$
\end{enumerate}
\end{lemma}
\begin{proof}
(1):\quad Is taken from \cite[Lemma 4.1]{RoSe2}, and is easy to calculate.
\medskip

\noindent
(2)
$\gr b=\gs b=(\gc a-\gl (c+y))b=\gc(\ga_b a+\gl x+\gl z)-\gl(c+y)b.$ So
\begin{align*}
&\ga_b(\gr-\gc)a+\gr c+(\gr-\gc\gl)x+\gr y+(\gr-\gc\gl)z\\
&=-\gl(c+y)(\ga_ba+c+x+y+z)\\
&=-\gl\Big(c^2+y^2+(cx+yz)+(cy+yc+\ga_b\gd y)+(cz+yx)\Big).
\end{align*}
Comparing the $\big(\ff a+\ff c\big)$-component we get (2).
\end{proof}

\begin{lemma}\label{Yoav1}$ $
If $\dim(A)=4$ and one of $x,y$ or $z$  is $0,$ then
$\gb_a=0.$
\end{lemma}
\begin{proof}
Assume first that $z= 0$. Then, by equation \eqref{eq bba},
\[
\gd by=\gd(\ga_ba+c+x+y)y=\gb_a(1-\gl')b+(\gl'-\ga_b)(\ga_b+\gd y).
\]
Hence (since $xy=0$)
\[
\gd(cy+y^2)=\gb_a(1-\gl')b+(\gl'-\ga_b)(\ga_ba+\gd y).
\]
In the LHS $b_{\gl,0}$ has coefficient $0.$  Hence, in the RHS
$\gb_a(1-\gl') x=0$.  Thus $\gb_a=0.$
\medskip

Assume next that $y= 0$. Then, by \eqref{4-2}, $x\ne 0.$
By  \eqref{eq bba},
\[
\gd b z=\gd(\ga_ba+c+x+z)z=\gb_a(1-\gl')b+(\gl'-\ga_b)(\ga_ba+\gd z).
\]
Hence (since $xz=0$)
\[
 \gd(\ga_b \gl z+cz+z^2)=\gb_a(1-\gl')b+(\gl'-\ga_b)(\ga_ba+\gd z).
\]
In the LHS $x$ has coefficient $0.$  Hence, in the RHS
$\gb_a(1-\gl') x=0$.  Thus $\gb_a=0.$
\medskip

Suppose $x=0.$  Then, by \eqref{4-2}, $y\ne 0.$ By \eqref{eq abb},
\[
\gl zb=\gl z(\ga_ba+c+y+z)=\gb_a(1-\gd')b+(\gd'-\ga_b)(\ga_ba+\gl z),
\]
so (since  $zy=0$),
\[
\gl(\ga_b\gd z+zc+z^2)=\gb_a(1-\gd')b+(\gd'-\ga_b)(\ga_ba+\gl z).
\]
In the RHS $y$ has coefficient $0$ so $\gb_a(1-\gd')=0,$ and again $\gb_a=0.$
\end{proof}

\begin{prop}\label{dim51}$
$
\begin{enumerate}
\item
$\big(\gb_a(1-\gd')+\gb_a(1-\gl')\big)x=\gl(1-\gd')x.$

\item
$\big(\gb_a(1-\gl')+\gb_a(1-\gd'))\big)y=\gd(1-\gl')y.$

\item
If $\gb_a=0,$ then $x=y=0.$

\item
If $\dim(A)\ge 4,$ then $x,y,z, x',y',z' $ are all distinct from $0.$

\item If $\dim(A)\ge 4,$ $\gl'=\gl,$ and $\gd'=\gd,$ then 
\[
\gl'=\gd'=\gl=\gd=2\ga_b=2\gb_a.
\]
\end{enumerate}
\end{prop}
\begin{proof}
(1) We have
\begin{equation}\label{x1}
\gb_a(1-\gl')x=\gd(yz+zy)\qquad \text{\rm (by\eqref{6-92}),}
\end{equation}
\begin{equation}\label{x2}
\Big(\gb_a(1-\gd')+\gl(\gd'-\ga_b)\Big)x=\gl(xc+zy)\qquad \text{(by \eqref{6-96})}.
\end{equation}
\begin{equation}\label{x3}
\ga_b\gl(\gl-1)x=(\gd-\gl)yz+\gd zy-\gl cx\qquad \text{(by Lemma \ref{bab}(2))}
\end{equation}
\begin{equation}\label{x4}
(1-\ga_b\gl) x=x c+ cx+yz+ zy\qquad \text{(by Lemma \ref{b^2}(2))}.
\end{equation}
So,
\begin{equation}\label{A}
\Big(\ga_b\gl(\gl-1)-\gb_a(1-\gl')\Big)x=-\gl yz-\gl cx\ \text{(by \eqref{x1} and \eqref{x3}).}\\
\end{equation}
Hence, using \eqref{x2},\eqref{A} and \eqref{x4}), we get
\[
\Big(\gb_a(1-\gd')+\gl(\gd'-\ga_b)\Big)x-\Big(\ga_b\gl(\gl-1)-\gb_a(1-\gl')\Big)x\\=\gl(1-\ga_b\gl)x.
\]
This shows (1).
\medskip

\noindent
(2)  We have
\begin{equation}\label{y1}
\Big(\gb_a(1-\gl')+\gd(\gl'-\ga_b)\Big)y=\gd(cy+xz)\qquad\text{(by \eqref{6-93})},
\end{equation}

\begin{equation}\label{y2}
\gb_a(1-\gd')y=\gl(xz+zx)\qquad\text{(by \eqref{6-97})},
\end{equation}

\begin{equation}\label{y3}
\ga_b\gd(\gd-1) y  =\gl xz +(\gl-\gd) zx  -\gd yc\qquad\text{(by Lemma \ref{bab}(3))},
\end{equation}

\begin{equation}\label{y4}
(1-\ga_b\gd)y=cy+yc+xz+zx\qquad\text{(by Lemma \ref{b^2}(3)))}.
\end{equation}
Hence by \eqref{y2} and \eqref{y3}
\[
\Big(\ga_b\gd(\gd-1)-\gb_a(1-\gd')\big)y=-\gd zx-\gd yc.
\]
From this and \eqref{y1} and \eqref{y4}, 
\[
\Big(\gb_a(1-\gl')+\gd(\gl'-\ga_b)\Big)y-\Big(\ga_b\gd(\gd-1)-\gb_a(1-\gd')\big)y=\gd (1-\ga_b\gd)y.
\]
This shows (2).
\medskip

\noindent
(3)  Suppose $\gb_a=0,$ then the LHS of both (1) and (2) is $0.$ 
So the RHS of both (1) and (2) is $0,$ and (3) follows from (1) and (2).
\medskip

\noindent
(4) Assume that $\dim(A)\ge 4.$  If $\dim(A)=5,$ then (4) follows
from Proposition \ref{span}.  Suppose $\dim(A)=4,$ and one of $x,y,z,x',y',z'$ 
is $0,$ say $x=0.$  By Lemma \ref{Yoav1}, $\gb_a=0,$ so by  (3), also $y=0.$  This contradicts
\eqref{4-2}.  The same argument shows that (4) holds if $y=0,$ or $z=0.$ By symmetry
and \eqref{4-1}, (4) holds.  
\medskip

\noindent
(5)  Assume that $\dim(A)\ge 4.$  By (4), $x$ and $y$ are not $0.$
Then, by (1) and~(2), 
since $\gl=\gl'$ and $\gd=\gd',$ we get $\gl=\gd.$
Now by (1) (or (2)), $\gb_a=\half\gl.$  Symmetrically using (4),
$\ga_b=\half\gl.$
\end{proof}

The following two results show that the hypothesis of Proposition \ref{dim51}(5) is impossible.

\begin{lemma}\label{f}
Suppose that $\gl=\gl'=\gd=\gd'=2\ga_b=2\gb_a.$ Then
\begin{enumerate}
\item
$\half\gl(b-a)+z-c=z'-c'.$

\item
$x-y=y'-x'.$

\item
$\half\gl a+c+x-y-z=\half\gl b+c'-x'+y'-z'.$
 
\item
\begin{align*}\tag{i}
yz+zy&=\half(1-\gl)x.\\\tag{ii}
xc+zy&=\half x.\\\tag{iii}
zy- cx&=\half\gl(\gl-1)x.\\\tag{iv}
x c+ cx+yz+ zy&=(1-\half\gl^2) x.
\end{align*}

\item
\begin{align*}\tag{i}
xz+zx&=\half(1-\gl)y.\\\tag{ii}
cy+xz&=\half y.\\\tag{iii}
xz  - yc&=\half\gl(\gl-1) y.\\\tag{iv}
yc+cy+xz+zx&=(1-\half\gl^2)y.
\end{align*}

\item
$zc=cz.$

\item
$x^2=y^2.$

\item
$c^2+y^2=\half(1+\gl)c.$ 

\item
$x^2+z^2=\half\gl(1-\half\gl)a+\half(1-\gl)c.$

\item
$z^2-c^2=\half\gl(1-\half\gl)a-\gl c.$
\end{enumerate}
\end{lemma} 
\begin{proof}
We have
\[
ab=\ga_b a+\gl x+\gl z=\gb_a b+\gl y'+\gl z',
\]
so
\begin{equation}\label{f1}
\half\gl(a-b)+\gl x+\gl z=\gl y'+\gl z'.
\end{equation}
By symmetry
\begin{equation}\label{f2}
\half\gl (b-a)+\gl x'+\gl z'=\gl y+\gl z.
\end{equation}
\medskip

\noindent
(1)\quad
Subtracting \eqref{f2} from \eqref{f1} we get
\begin{align*}
\Big(\half\gl(a-b)+\gl x+\gl z\Big)&-\Big(\half\gl (b-a)+\gl x'+\gl z'\Big)\\
&=\gl(y'+z'-y-z),
\end{align*}
so,
\[
a-b+x+y+2z=x'+y'+2z',
\]
or
\[
a-b+(b-\half\gl a-c)+z=a-\half\gl b-c'+z'.
\]
This shows (1).
\medskip

\noindent
(2)\quad  This is obtained by adding \eqref{f1} to \eqref{f2}.
\medskip

\noindent
(3)  This is obtained by subtracting  (1) from  (2).
\medskip

\noindent
(4\& 5)  Part (4) are equation \eqref{x1}--\eqref{x4} in this special case.  Part (5) 
are equation \eqref{y1}--\eqref{y4} in this special case.
\medskip

\noindent
(6)  This follows from Lemma \ref{bab}(4).
\medskip

\noindent
(7)  This follows from Lemma \ref{bab}(1).
\medskip

\noindent
(8)  This follows from Lemma \ref{6}(2), because here $\gr=\gc=\half(\gl(1-\gl)-\gl.$
\medskip

\noindent
(9) $c^2+x^2+y^2+z^2=\half\gl(1-\half\gl)a+c,$ by Lemma \ref{b^2}(1).
Now (9) follows from (8).
\medskip

\noindent
(10) By (9),
\[
c^2+x^2+z^2-c^2=\half\gl(1-\half\gl)a+\half(1-\gl)c.
\]
By (8) and (9) we get
\[
z^2-c^2=\half\gl(1-\half\gl)a+\half(1-\gl)c-\half(1+\gl)c=\half\gl(1-\half\gl)a-\gl c.\qedhere
\]
\end{proof}

\begin{prop}\label{f+}
Assume the hypotheses of Lemma \ref{f} hold.  Then
\begin{enumerate}
\item
$z'=\half (b-ba)+\frac{bz-bc}{\gl}.$

\item 
$c'=z'+\half\gl (a-b)-z+c.$

\item
$(z')_{\gl,0}=0,$ $z'_{0,\gl}=0,$
$z'_{\gl,\gl}=\half z,$ $\ga_{z'}a +z'_{0,0}=\half (1-\half\gl) a- \half c,$ so
\[\tag{$*$}
z'=\half (1-\half\gl) a- \half c+\half z.
\]

\item
$c'_{\gl,0}=-\half\gl x$  and $c'_{0,\gl}=-\half\gl y.$

\item
$\dim(A)\le 3.$
\end{enumerate}
\end{prop}
\begin{proof}
(1)  This is obtained from Lemma \ref{f}(1) by multiplying
on the left by $b$ and then dividing by $\gl.$
\medskip

\noindent
(2) This comes from Lemma \ref{f}(1)
\medskip

\noindent
(3) By (1) and by Lemma \ref{f} 4(i) and 4(ii),
\[
\textstyle{z'_{\gl,0}=\half x+\frac{yz-xc}{\gl}=\half x-\half x=0.}
\]
By (1), and Lemma \ref{f}(5(iii)),
\[
z'_{0,\gl}=\half(1-\gl)y+\frac{xz-yc}{\gl}=\half(1-\gl)y+\half(\gl-1)y=0.
\]
By (1) and Lemma \ref{f}(5),
\[
z'_{\gl,\gl}=\half(1-\gl)z+\frac{\half\gl^2z+cz-zc}{\gl}=\half(1-\gl)z+\half\gl z=\half z.
\]
Finally, by (1) and Lemma \ref{f}(10),
\begin{gather*}
\ga_{z'}a+z'_{0,0}=\half c+\frac{z^2-c^2}{\gl}=\half c+\half(1-\half\gl)a-c\\
=\half (1-\half\gl) a- \half c.
\end{gather*}
\medskip

\noindent
(4) This follows from (2) and (3).
\medskip

\noindent
(5) Note that by (4) and by symmetry, $c_{\gl,0}=-\half \gl x'.$
By (3), and by symmetry,
$z\in \ff b+\ff c'+\ff z'.$ Comparing the $(\gl,0)$-component in the two sides of 
equality $(*)$ for $z'$ in (3),
with respect to $b,$ we get
\[
0=(\half\gl-1)a_{\gl,0}+c_{\gl,0}=(\half\gl-1)x'-\half\gl x'=-x'.
\]
By Proposition \ref{dim51}(4), $\dim(A)\le 3.$
\end{proof}

\subsection{\bf The dimension of $A$ is $\le 3$}\label{main1}$ $
\medskip

\noindent
The goal of this subsection is to prove that $\dim(A)\le 3.$
So we assume, toward a contradiction, that $\dim(A)\ge 4.$
As mentioned in the introduction, we know already that $\dim(A)\le 5.$

\begin{prop}\label{dim4} 
Consider the assertion
\[\tag{$*$}
\begin{aligned} 
&\text{there exist a Miyamoto involution $\gt$ associated with $b$} \\
&\text{such that $A$ is generated by $a, a^{\gt}.$}
\end{aligned}
\]
\begin{enumerate}
\item
If $\dim(A)=4,$ then $(*)$ holds.

\item
If $\dim(A)=5,$ and any subalgebra of $A$ generated by two primitive axes
of the same type has dimension distinct from $4,$  then $(*)$ holds.
\end{enumerate}
 \end{prop}
  \begin{proof} 
	Suppose that $\dim(A)\ge 4.$ By Proposition \ref{dim51}(4), $x',y',z'$ are all distinct from $0.$
Hence there are  three distinct Miyamoto involutions 
$\gt_1,$ $\gt_2,$ $\gt_3,$ with respect to $b$, and thus axes of the form 
$a^\gt$ where $\gt = \gt_i,$ $1 \le i \le 3,$ and the 
axes $a,a^{\gt_1},a^{\gt_2},a^{\gt_3}$ are linearly independent.

Let $a'$ be one of the axes of the form
$a^\gt.$ Suppose that $A',$ the subalgebra generated by $a,a^{\gt},$
has dimension $\le 3.$

If $\gl\ne \gd$ then   since $a'$ is also of type 
$(\gl,\gd),$ $A'$ is  as in Remark  \ref{eg}(1) (with $a'$ in place of $b$).
Writing 
\[
a' = \gc_a a+\gc_c c +\gc_x x + \gc_y y + \gc_z z,
\]
We have
\[
aa' =  \gd a + \gl a' = (\gd +\gl \gc_a)a+ \gl \gc_c c   +\gl \gc_x
   x +\gl \gc_y y +\gl \gc_z z.
\]
 On the other hand, 
\[
a a' = a(   \gc_a a+\gc_c c 
  + \gc_x x + \gc_y y + \gc_z z)= \gc_a a+\gl \gc_x x +\gl \gc_z z.
	\]
Matching components shows $\gc _c = 0,$ $\gd +\gl \gc_a = \gc_a$,
and $\gc_y = 0.$  Also symmetrically $\gc_x = 0$ (considering $a'a = \gl a +
\gd a'$), implying 
\[
a' = a + \gc_z z,
\]
so $a' \in \ff a + \ff z.$ But there
 are three linearly independent possibilities for the axes
$a^\gt,$ a contradiction.

Thus we may assume that $\gl = \gd,$ which, by Remark \ref{eg}, implies $A'$ is commutative. But now
$aa' = a'a$ implies $\gc_x = \gc_y = 0,$ so $a' \in  \ff a+\ff c  +
\ff z,$ again contradicting $a,a^{\gt_1},a^{\gt_2},a^{\gt_3}$
being linearly independent.

Hence $\dim(A')\ge 4,$ for some $\gt\in\{\gt_1,\gt_2,\gt_3\}.$
This shows that both parts (1) and (2) hold.
\end{proof}

\begin{thm}\label{dim401}
$\dim(A)\le 3.$
\end{thm}
\begin{proof}
Suppose first that $\dim(A)=4.$  By Proposition \ref{dim4}(1), 
we may assume that $b$ is also of type $(\gl,\gd).$
By Proposition \ref{dim51}(5), $\gl=\gl'=\gd=\gd'=2\ga_b=2\gb_a.$
Now Proposition \ref{f+}(5) supplies a contradiction.

Suppose $\dim(A)=5.$  Then, by the previous paragraph, the hypothesis
of Proposition \ref{dim4}(2) holds, so by that proposition
we may assume that $b$ is also of type $(\gl,\gd)$ and the same
argument leads to a contradiction.
\end{proof}
\medskip

\noindent
\begin{proof}[Proof of the Main Theorem.]
We can now complete the proof of the Main Theorem.
In subsection \ref{main1} we showed that part (1) of the Main Theorem holds.
It remains to prove part (2).  Let $X$ be the set
of {\it all} primitive axes in $A.$ 
By~\cite[Theorem 3.3]{RoSe2}, $A$ is spanned by $X.$
Let $a,b\in X,$  with $a$ of type $(\gl,\gd).$ By part (1) of the Main Theorem,
$\dim(B)\le 3,$ where $B$ is the subalgebra of $A$ generated by $a,b.$
By \cite[Theorem C]{RoSe2} and \cite[Theorems A, B]{RoSe1},
$a,b$ are of Jordan type {\it inside the subalgebra $B$.}
Thus $b\in\ff a+A_{0,0}(a)+A_{\gl,\gd}(a).$  Since this
holds for any $b\in X,$ and since $A$ is spanned
by $X,$ we see that $A=\ff a+A_{0,0}(a)+A_{\gl,\gd}(a).$
Thus $a$ has Jordan type.  As $a$ was chosen arbitrarily,
we see that this holds  for all $x\in X.$
\end{proof}

\end{document}